\documentclass[]{elsarticle}
\usepackage{lineno}
\usepackage{bbold}
\usepackage{verbatim}
\usepackage{url}
\usepackage{graphicx}
\usepackage{float}
\usepackage{mathrsfs}
\usepackage{epstopdf}
\usepackage[all,color]{xy}
\usepackage{cancel}
\usepackage{MnSymbol}
\usepackage{tikz}
\usepackage[margin=1.5in]{geometry}

\usetikzlibrary{arrows}
\modulolinenumbers[5]
\allowdisplaybreaks
\newtheorem{theorem}{Theorem}

\newtheorem{example}[theorem]{Example}
\newtheorem{lemma}[theorem]{Lemma}
\newtheorem{proposition}[theorem]{Proposition}

\newproof{proof}{Proof}
\numberwithin{equation}{subsection}
\numberwithin{theorem}{subsection}

\newcommand{\ignore}[1]{}

\DeclareMathOperator{\asgn}{asgn}
\DeclareMathOperator{\csgn}{csgn}

\begin{document}

\begin{frontmatter}

\title{Oriented hypergraphs and generalizing the Harary-Sachs theorem to integer matrices.}

\author[add2]{Blake Dvarishkis} 
\author[add3]{Josephine Reynes}

\author[add2]{Lucas J. Rusnak\corref{mycorrespondingauthor}}\ead{Lucas.Rusnak@txstate.edu}

\address[add2]{Department of Mathematics, Texas State University, San Marcos, TX 78666, USA}
\address[add3]{Department of Mathematics, Norwich University, Barre, VT, 05641, USA}

\cortext[mycorrespondingauthor]{Corresponding author}

\begin{abstract}

Incidence-based generalizations of cycle covers, called contributors, extend the Harary-Sachs coefficient theorem for characteristic polynomials of the adjacency matrix of graphs. All minors of the Laplacian resulting from an integer matrix are characterized using their associated oriented hypergraph through a new minimal collection of contributors to produce the coefficients of the total-minor polynomial. We prove that the natural grouping of contributors via tail-equivalence is necessarily cancellative for any contributor family that reuses an edge. We then provide a new combinatorial proof on the non-$0$ isospectrality of the traditional characteristic polynomials of the Laplacian and its dual.
\end{abstract}


\end{frontmatter}


\section{Introduction and Background}

The Harary-Sachs Theorem determines the coefficients for the characteristic polynomial of adjacency matrices of graphs through signed weighted sums of certain subgraphs \cite{SGBook, harary1962determinant, sachs1966teiler}. A uniform hypergraph generalization by Clark and Cooper appears in $\cite{CLARK2022354}$, while a signed graph version for the Laplacian appears in \cite{Sim1} by Belardo and Simi{\'c}. These results have since been shown to have a unifying context using oriented hypergraphs associated to any integer matrix using a finer approach needing only signed configurations. An oriented hypergraphic characterization of both the adjacency and Laplacian matrices appears in \cite{OHSachs}, the connection to reclaiming the matrix-tree theorem for both graphs and signed graphs appears in \cite{OHMTT}, generalizing Chaiken from \cite{Seth1}. Finally, a complete oriented hypergraphic classification of all minors of the adjacency and Laplacian matrices arising from integer matrices appeared in \cite{IH2} with the creation of the total-minor polynomial.

We prove an improved version of the total-minor polynomial by exhibiting a new smallest class of hypergraphic mapping families necessary to determine all Laplacian ordered cofactors. This is accomplished by adapting the techniques from \cite{OHHada}, which produces a new formula for Hadamard's Maximum Determinant Problem for $\{\pm 1\}$-matrices, to the total minor polynomial introduced in \cite{IH2}. We also introduce the concept of dual mapping families which allows for a new combinatorial proof that the spectrum for any Laplacian and its dual agree up to non-zero eigenvalues. Moreover, this indicates that similar results do not hold for off-diagonal minors.

\subsection{Oriented Hypergraphs}

The connection between integer matrices and oriented hypergraphs was introduced in \cite{AH1, OH1} where signed graphic techniques were applied locally between adjacencies to study the matrix structure through hypergraphic families. This was done by adapting the bidirected graph interpretation of orientations of signed graphs \cite{MR0267898,SG,OSG} to multi-directed hypergraphs and building on the the concept of balanced matrices \cite{BM2,DBM,BM} to provide new techniques to bridge graph theory and matrix theory. A generalized algebraic hypergraphic theory has since emerged for signed adjacency and Laplacian matrices, including spectral properties \cite{Reff6,Reff7,Mulas3,Mulas1,Reff2}, characteristic polynomials and minors \cite{OHSachs,IH2}, matrix-tree-type theorems \cite{OHMTT, SGKirch}, and Hadamard matrices \cite{Reff3,Reff5,OHHada}.

An \emph{oriented hypergraph} is a sextuple $G=(V, E, I, \varsigma, \omega, \sigma)$ where $V$, $E$ and $I$ are the sets of vertices, edges, and incidences, respectively; while $\varsigma: I \rightarrow V$, $\omega: I \rightarrow E$, and $\sigma:I\rightarrow \{+1,-1\}$ are the incidence port, attachment, and signing functions, respectively. A value of $\sigma(i) = +1$ is indicated by an arrow at incidence $i$ entering the vertex, while a value of $\sigma(i) = -1$ is indicated by an arrow at incidence $i$ exiting the vertex. The \emph{incidence dual} of an oriented hypergraph $G=(V,E,I,\varsigma,\omega, \sigma)$ is the oriented hypergraph $G^*=(E,V,I,\omega,\varsigma, \sigma)$. Incidence duality is the primary duality for incidence hypergraphs and replaces the line graph. The need for this level of specificity for graph-like objects has been addressed in \cite{IH1, IH2} where the categorical deficiencies in both graph and set-system hypergraphs were shown to not be robust enough to examine integer matrices.

An oriented hypergraph in which each edge has exactly two attachments is called a \emph{bidirected graph}. A bidirected graph can be regarded as an orientation of a signed graph where edge $e$ is signed $-\sigma(i)\sigma(j)$ where incidences $i$ and $j$ are the edge-ends of edge $e$, see \cite{MR0267898, SG, OSG} for the connection between signed graphs and bidirected graphs. Figure \ref{fig:OHs} depicts a bidirected graph $G_1$ and a ``multi-directed'' oriented hypergraph $G_2$.

\begin{figure}[H]
    \centering
    \includegraphics{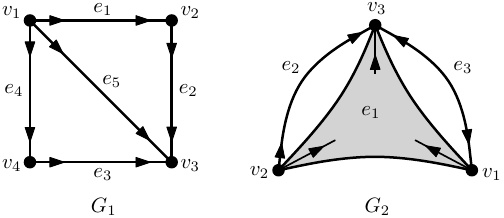}
    \caption{Two oriented hypergraphs.}
    \label{fig:OHs}
\end{figure}

Many oriented hypergraphic definitions coincide with their locally-signed-graphic embedding \cite{AH1, OH1}. This notion of treating hypergraphs locally as signed graphs was formalized in \cite{IH1, IH2} and  \emph{locally-signed-graphic} simply means the oriented hypergraph is studied via the signed adjacencies of a hyperedge. 

A \emph{directed path of length $n/2$} is a sequence of distinct vertices, edges, and their intermediary incidences 
\begin{equation*}
\overrightarrow{P}_{n/2}=(a_{0},i_{1},a_{1},i_{2},a_{2},i_{3},a_{3},...,a_{n-1},i_{n},a_{n})
\end{equation*}
where $\{a_k\}$ is an alternating sequence of vertices and edges, and $i_{k}$ is an incidence between $a_{k-1}$ and $a_{k}$. A \emph{path of length $n/2$} is the set on the directed path of length $n/2$ sequence. A \emph{directed circle} is a closed directed path, while a \emph{circle} will be the set on the directed circle. 

A \emph{directed adjacency of $G$} is an incidence-monic map of $\overrightarrow{P}_{1}$ into $G$; that is, it does not reuse an incidence. Due to the nature of the incidence-maps, it is possible for $\overrightarrow{P}_{1}$ to fold back on itself creating a \emph{backstep} of the form $v,i,e,i,v$ --- the number of backsteps correspond to the degree of $v$. The \emph{adjacency sign of a path} $P$ is 
\begin{equation*}
asgn(P)=(-1)^{\lfloor n/2\rfloor }\prod_{k=1}^{n}\sigma (i_{k})\text{,}
\end{equation*}
which is equivalent to taking the product of the signed adjacencies if $P$ is a vertex-path. 

Returning to Figure \ref{fig:OHs} we see that every adjacency (edge) in $G_1$ is positive as the arrows in a given edge have opposite signs thus $-\sigma(i)\sigma(j)$ is positive --- this is most easily seen as the arrows travel coherently through the edge. Signed graphs with all adjacencies positive are indistinguishable from ordinary graphs. An adjacency is negative if the arrows represent the same sign, so they are visually opposing each other in the oriented hypergraph --- in $G_2$ the $(v_1 , v_3)$-adjacency using edge $e_3$ and the $(v_1 , v_2)$-adjacency using edge $e_1$ are both negative.

The \emph{incidence matrix} of an oriented hypergraph $G$ is the $V \times E$ matrix $\mathbf{H}_{G}$ where the $(v,e)$-entry is the sum of $\sigma(i)$ for each $i \in I$ with $\varsigma(i) = v$ and $\omega(i) = e$. A bidirected graph in which every edge/adjacency is positive is regarded as an orientation of an ordinary graph as they have indistinguishable incidence matrices. The \emph{Laplacian matrix of $G$} is defined as $\mathbf{L}_{G}:=\mathbf{H}_{G} \mathbf{H}_{G}^{T}=\mathbf{D}_{G}-\mathbf{A}_{G}$ for all oriented hypergraphs \cite{AH1}.

\begin{example}
The incidence and Laplacian matrices for $G_1$ and $G_2$ from Figure \ref{fig:OHs} are as follows:

\begin{align*}
    \mathbf{H}_{G_1}=
\left[ 
\begin{array}{ccccc}
-1 & 0 & 0 & -1 & -1\\ 
1 & -1 & 0 & 0 & 0\\ 
0 & 1 & 1 & 0 & 1\\ 
0 & 0 & -1 & 1 & 0 \\
\end{array}
\right]
& &
\mathbf{H}_{G_2}=
\left[ 
\begin{array}{ccccc}
-1 & 0 & 1\\ 
-1 & -1 & 0\\ 
1 & 1 & 1\\ 
\end{array}
\right] \\
\mathbf{L}_{G_1}=
\left[ 
\begin{array}{ccccc}
3 & -1 & -1 & -1\\ 
-1 & 2 & -1 & 0\\ 
-1 & -1 & 3 & -1\\ 
-1 & 0 & -1 & 2
\end{array}
\right]
& &
\mathbf{L}_{G_2}=
\left[ 
\begin{array}{ccccc}
2 & 1 & 0\\ 
1 & 2 & -2\\ 
0 & -2 & 3 
\end{array}
\right]
\end{align*}

Note that any integer matrix can be regarded as the incidence matrix of an oriented hypergraph with the inclusion of multiple arrows at vertices.
\end{example}

It is a simple and well-known fact that the traditional characteristic polynomials of the Laplacian and its dual are identical up to $0$ eigenvalues. For example, the traditional characteristic polynomials of $G_1$ from Figure \ref{fig:OHs} and its dual are

\begin{align}
   \chi(\mathbf{L}_{G},x) = x^4 - 10x^3 +32x^2 -32x \label{eq:1}
\end{align}

and

\begin{align}
   \chi(\mathbf{L}_{G^*},x) = x^5 - 10x^4 +32x^3 -32x^2. \label{eq:2}
\end{align}
The coefficients also track diagonal minors with size determined by the degree of the monomial. 

In \cite{IH2} the total minor polynomial was introduced to study all ordered minors. Let $\mathbf{X}$ be the $V \times V$ matrix whose $ij$-entry is $x_{ij}$, the \emph{total-minor polynomial} of $\mathbf{L}$ is the multivariable characteristic polynomial $\chi (\mathbf{L},\mathbf{x}):=\det (\mathbf{X-L)}$. We improve the summation when determining the total minor polynomial to an optimal collection of hypergraph maps that are necessarily edge monic. We then use this characterization to revisit the non-$0$ isospectrality of the traditional characteristic polynomials of the Laplacian and its dual. We demonstrate that this phenomenon is actually a duality relationship that resorts positionally-labeled monomials among different minor families in the total minor polynomial. Moreover, this duality is not expressed across all minors.

\subsection{Unifying Matrix-tree and Harary-Sachs Theorems}

The Harary-Sachs Theorem uses vertex covers to build permutation analogs to study the characteristic polynomial of the adjacency matrix of a graph \cite{SGBook}, and the signed graphic version appears in \cite{Sim1}. The oriented hypergraphic version from \cite{OHSachs} extends this to both the adjacency and Laplacian matrices of oriented hypergraphs, while the all-minors version appears in \cite{IH2}. These permutation-analogs, called \emph{contributors}, are used in these generalizations and allow for path images to fold back on themselves reusing the same incidence. 

Formally, a \emph{contributor of $G$} is an incidence preserving map from a disjoint union of $\overrightarrow{P}_{1}$'s with tail $t$ and head $h$ into $G$ defined by $c:\dcoprod \limits_{v\in V}\overrightarrow{P}_{1}\rightarrow G$ such that $c(t_{v})=v$ and $\{c(h_{v})\mid v\in V\}=V$. As the set of heads and the set of tails of a contributor both cover $V$, they are naturally associated with a permutation. However, many contributors may be associated to the same permutation. 

An early version of contributor sorting appeared in \cite{OHMTT} and formalized in \cite{IH2} where contributors are sorted by their first steps (their tail images in $G$) into \emph{tail-equivalence classes}. It was shown in \cite{OHMTT} that when $G$ is $2$-uniform (is a signed graph) each tail-equivalence class is a Boolean lattice ordered by circle containment. This leads to a refinement of Chaiken's All-minors Matrix-tree Theorem for signed graphs in \cite{Seth1}.

\textbf{Notation:} The set of contributors of an oriented hypergraph is denoted ${\mathfrak{C}}(G)$. Throughout, let $U,W \subseteq V$ with $\lvert U \rvert = \lvert W \rvert$, while a total ordering of each set will be denoted by $\mathbf{u}$ and $\mathbf{w}$, respectively. Let ${\mathfrak{C}}(G;\mathbf{u},\mathbf{w})$ be the set of \emph{restricted} contributors in $G$ where $c(u_i)=w_i$, and two elements of ${\mathfrak{C}}(G;\mathbf{u},\mathbf{w})$ are said to be \emph{$[\mathbf{u},\mathbf{w}]$-equivalent}. Let $\mathfrak{\hat{C}}(G;\mathbf{u},\mathbf{w})$ be the set obtained by removing the pairwise $\mathbf{u} \rightarrow \mathbf{w}$ mappings from ${\mathfrak{C}}(G;\mathbf{u},\mathbf{w})$; the elements of $\mathfrak{\hat{C}}(G;\mathbf{u},\mathbf{w})$ are called the \emph{reduced} $[\mathbf{u},\mathbf{w}]$-equivalent contributors. To avoid confusion between an algebraic cycle and a graph component that forms a closed walk, we reserve the term ``cycle'' for algebraic cycles of permutations.

\begin{example}
Figure \ref{fig:Conts} depicts two contributors, one in $G_1$ and one in $G_2$. Note that the incidence arrows are omitted, and the direction of the each $\overrightarrow{P}_{1}$ map is shown.

The $G_1$ contributor is naturally associated to the permutation $(v_1 v_2 v_3)(v_4)$, where the backstep at $v_4$ represents the fixed point. Note that a backstep reuses the same incidence and are distinct from loops which use different incidences but also represent fixed points. The ability to distinguish between loops and backsteps is critical to the generalization from adjacency matrices to Laplacian matrices, as discussed in \cite{AH1}.

\begin{figure}[H]
    \centering
    \includegraphics{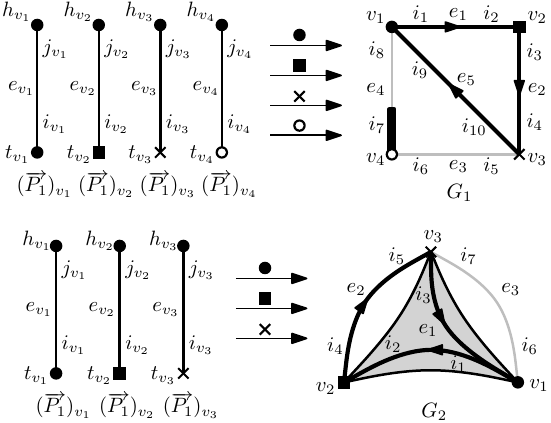}
    \caption{Two contributor examples.}
    \label{fig:Conts}
\end{figure}

The $G_1$ contributor is also a member of ${\mathfrak{C}}(G_1;v_1,v_2)$ as it has a $v_1 \rightarrow v_2$ path image (adjacency). It is also a member of ${\mathfrak{C}}(G_1;(v_1,v_4),(v_2,v_4))$ as it has both a $v_1 \rightarrow v_2$ path image and a $v_4 \rightarrow v_4$. Removing these would produce reduced contributors.

The $G_2$ contributor is naturally associated to the permutation $(v_1 v_2 v_3)$. This contributor is an example of a special type of contributor we will analyze in Section \ref{sec:EM}, called \emph{non-edge-monic}, that is, an edge of $G_2$ appears in more than one $\overrightarrow{P}_{1}$ image.
\end{example}

The hypergraph $G^0$ is the \emph{zero-loading of $G$} and extends the hypergraph to a uniform incidence hypergraph while assigning a weight of $0$ to all new incidences. The value of $\asgn(c)$ is adjusted to include a value of $0$ for $\sigma$. Thus, a reduced contributor exists in $G$ if and only if it is non-zero. As discussed in \cite{IH2,SGKirch}, the $\mathbf{u} \rightarrow \mathbf{w}$ maps need not exist in $G$ if they are removed, but the maps must be allowed to exist a priori their removal in order to define everything --- this is remedied by the zero-loading $G^0$ to allow all possible maps to exist, and the non-zero ones are those that actually ``contribute'' to the minors. To simplify notation, let $\mathfrak{\hat{C}}_{\neq0}(G^0;\mathbf{u},\mathbf{w})$ be the set of non-zero reduced $[\mathbf{u},\mathbf{w}]$-equivalent contributors in $G^0$; that is, the reduced $[\mathbf{u},\mathbf{w}]$-equivalent contributors that reside in $G$.

Let $ec(c)$, $oc(c)$, $pc(c)$ and $nc(c)$ be the number of even, odd, positive, and negative components in a (reduced-)contributor $c$, respectively, as determined by number of adjacencies or $\asgn$. Also, let $bs(c)$ denote the number of backsteps in contributor $c$. It is worth noting that backsteps are technically negative weak walks that do not arise from adjacencies. Backsteps are treated separately to help the reader as they are essentially ``hole-free loops'' but they are never ``subobjects,'' thus necessitating the need for the zero-loading and the functional definition of contributor; specific details can be found in \cite{IH2} which rectifies the lack of a categorical subobject classifier for graphs by transitioning to oriented hypergraphs. Finally, $ec(\check{c})$ is the number of even cycles in the un-reduced contributor before it was reduced to $c$. Define the \emph{contributor sign} as
\begin{align*}
    \csgn(c)=(-1)^{ec(\check{c})+nc(c)+bs(c)}
\end{align*}
and we have the following restatement of the Total-minor Polynomial Theorem from \cite{IH2}.
\begin{theorem}[Total-minor Polynomial \cite{IH2}]
\label{t:Main}
Let $G$ be an oriented hypergraph with Laplacian matrix $\mathbf{L}_{G}$. Then, 

\begin{align*}
    \chi(\mathbf{L}_{G},\mathbf{X}) = \dsum\limits_{ [\mathbf{u},\mathbf{w}]}\left( \dsum\limits_{\substack{c \in \mathfrak{\hat{C}}_{\neq0}(G^0;\mathbf{u},\mathbf{w}) \\ }} \csgn(c) \right) \dprod\limits_{i}x_{u_{i},w_{i}}.
\end{align*}
\end{theorem}

\begin{example}
    The coefficient of $x_{v_2 v_2}$ is the ordered $(v_2,v_2)$-cofactor of $\mathbf{L_{G_1}}$. This is calculated by considering all contributors of $G_1$ restricted to $v_2 \rightarrow v_2$, then remove the $v_2 \rightarrow v_2$ map and sign each contributor via $\csgn$. The $8$ non-cancellative contributors are depicted in Figure \ref{fig:v2v2l}. Observe that the natural extension of these backsteps across their indicated edge produce the $8$ spanning trees from the Matrix-tree theorem, and they are all directed toward sink $v_2$.
\begin{figure}[H]
    \centering
    \includegraphics{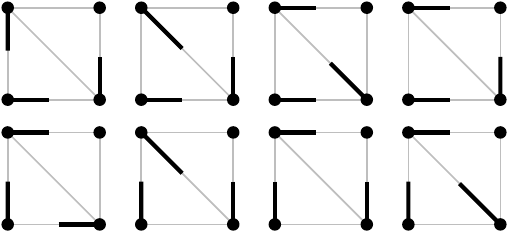}
    \caption{All $8$ non-cancellative reduced contributors of $G_{1}$ where $v_2 \rightarrow v_2$ is removed}
    \label{fig:v2v2l}
\end{figure} 
Note, there are more contributors where $v_2 \rightarrow v_2$. However, these all naturally cancel, as discussed in \cite{OHMTT}, which we expand upon in Section \ref{sec:EM} to prevent large brute-force calculations of contributor signs.
\end{example}

\section{Contributor Equivalence}
\label{sec:EM}

\subsection{Tail-equivalence}
\label{ssec:TEC}

Brute-force calculations of contributor signs can be mitigated by the use of a contributor-partitioning process via ``tail-equivalence'' classes. Two contributors are said to be \emph{tail-equivalent} if the image of their tail-incidences are the same, and \emph{head-equivalent} if the image of their head-incidences are the same. These create a new partition of contributors in \emph{tail-equivalence classes} and \emph{head-equivalence classes}. Each tail-equivalence class (and head-equivalence classes, by duality) is uniquely represented by its identity-permutation representative --- that is, a contributor consisting of only backsteps. Tail-equivalence was first observed in \cite{OHMTT} under the name ``activation class'' for signed/bidirected graphs before it was extended to contributors of oriented hypergraphs in \cite{IH2}, as ``activation'' was not well-defined on hyperedges. There is a trivial relationship between tail-equivalence classes and head-equivalence classes:

\begin{lemma}
    For any oriented hypergraph $G$ the set of tail-equivalence classes can be re-sorted into (different) head-equivalence classes by reversing all cycles, which naturally correspond to their algebraic inverse permutation (there are still multiple copies of the permutations for each respective contributor). 
\end{lemma}

However, there are some rather strong established properties of tail-equivalence classes of bidirected graphs (i.e. all edge sizes are at most $2$) that we re-present here:

\begin{theorem}[\cite{OHMTT}]
    All tail-equivalence classes of a bidirected graph are Boolean lattices ordered by circle containment.
\end{theorem}

For ordinary graphs, each Boolean lattice tail-equivalence class alternates sign by rank, and is a very quick combinatorial way to see that $\det(\mathbf{M}_G) = 0$.

\begin{theorem}[$k$-arborescences \cite{IH2}]
    Every single element tail-equivalence class of $k$-reduced contributors of a bidirected graph is unpacking equivalent to $k$-arborescences.
\end{theorem}

\begin{example}
Again, for ordinary graphs this provides a new combinatorial interpretation of the Matrix-tree Theorem. Figure \ref{fig:v2v2l} depicts the non-cancellative (single-element) Boolean lattices for the $(v_2,v_2)$-cofactor.

Three tail-equivalence classes for oriented hypergraph $G_1$ are shown in Figure \ref{fig:tec2} within each circle. Those contributors that have a $v_2 \rightarrow v_2$ map are dashed, and those that do not have a $v_2 \rightarrow v_2$ are greyed out.

\begin{figure}[H]
    \centering
    \includegraphics{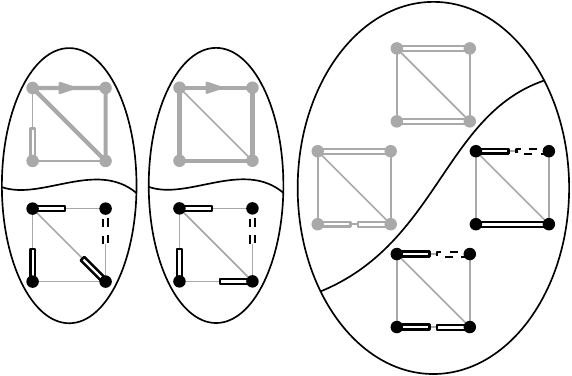}
    \caption{Three tail-equivalence classes and $v_2$-reduced tail-equivalence classes for $G_1$.}
    \label{fig:tec2}
\end{figure} 

Notice that the backsteps of the single element ($v_2,v_2$)-reduced tail-equivalence classes naturally extend into spanning trees directed toward $v_2$.
\end{example}

While Theorem \ref{t:Main} holds for all integer matrix Laplacians through their associated oriented hypergraph, there is not a Boolean structure to rely upon.

\begin{lemma}
    An incidence-simple single-edge hypergraph on $n$ vertices $E_n$ contains $n!$ contributors that follow the Stirling numbers of 1st kind.
\end{lemma}

\begin{example}
There are a total of $12$ tail-equivalence classes in $G_2$ with their identity permutation representative consisting of only backsteps in Figure \ref{fig:tec3}. tail-equivalence class (1) has all backsteps into a common $3$-edge, hence it will have $3!$ contributors.
\begin{figure}[H]
    \centering
    \includegraphics{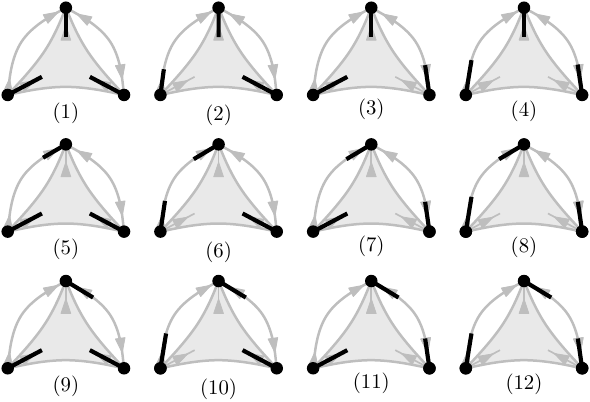}
    \caption{There are twelve tail-equivalence classes for $G_2$. Each represented by their identity-permutation representative.}
    \label{fig:tec3}
\end{figure} 

However, other identity-contributor representatives reside in some, or all, of the edges. Essentially, destroying the simplicity of having Boolean or Stirling counts. We will prove in Section \ref{ssec:emtec} that those tail-equivalence classes that have any backsteps into a common edge will necessarily cancel.
\end{example}

\subsection{Edge-Monic Contributors}
\label{ssec:emtec}

A contributor $c$ is \emph{edge-monic} if $c(e_{v}) \neq c(e_{w})$ for all distinct $v,w \in V$. We will denote the set of edge-monic contributors as $\mathfrak{M}(G)$, with ${\mathfrak{M}}(G;\mathbf{u},\mathbf{w})$, and $\mathfrak{\hat{M}}(G;\mathbf{u},\mathbf{w})$ defined as the edge-monic analogs of \emph{restricted} and \emph{reduced} contributors. Similarly, let $\mathfrak{N}(G)$ denote the set of non-edge-monic contributors. A tail-equivalence class (head-equivalence class) is \emph{edge-monic} if every contributor in it is edge-monic. We have the following trivial observation.

\begin{lemma}
Every contributor in a tail-equivalence class (head-equivalence class) is edge-monic if, and only if, one contributor in the class is edge-monic.
\label{l:TECEM}
\end{lemma}

\begin{example}
    It is easily checked that of the contributors in Figure \ref{fig:tec3} only (4), (7), and (10) are edge monic.
\end{example}

\begin{theorem}
    Let $N$ be a non-edge-monic tail-equivalence (or head-equivalence) class of an oriented hypergraph $G$, then
    \begin{align*}
        \dsum_{c \in N} \csgn(c) = 0.
    \end{align*}
    \label{t:contcancel}
\end{theorem}
\begin{proof}
    Suppose $v_{1}$ and $v_{2}$ have their tails (heads) in the same edge. Since their tails (heads) are in the same edge, right multiplication by the 2-cycle $(12)$ reverses the head images, thus, preserving the tail-equivalence (head-equivalence) classes and binning $v_{1}$ and $v_{2}$ if they started separate, or separating them if they started together. Since transpositions are their own inverse, there are exactly half the contributors in the tail-equivalence class where $v_{1}$ and $v_{2}$ are in separate cycles. Consider these three cases:
    
\textit{Case 1:} Let $v_{1}$ and $v_{2}$ both be in backsteps. Consider the permutations for these types of contributors. Right multiply these permutations by the 2-cycle $(12)$ corresponding to $v_{1}$ and $v_{2}$. This turns the $2$ backsteps into a 2-cycle, which shows up graphically as a repeated adjacency due to the non-edge-monicness. The 2 backsteps are offset by the two negative signs in the new adjacencies. The final sign changes due to the inclusion of a new even cycle.

\textit{Case 2:} Suppose only one of $v_{1}$ and $v_{2}$ are in a backstep. We will go with $v_{1}$ being in a backstep and $v_{2}$ in a non-trivial cycle. Consider the permutations for these types of contributors. Right multiply by the 2-cycle $(12)$ corresponding to $v_{1}$ and $v_{2}$. The loss of a backstep is offset by the creation of a new adjacency. The final sign changes due to the loss/gain of an even cycle.

\textit{Case 3:} Let $v_{1}$ and $v_{2}$ be in disjoint non-trivial cycles of length $m$ and $n$ respectively. Consider the permutations for these types of contributors. Right multiply by the 2-cycle $(12)$ corresponding to $v_{1}$ and $v_{2}$. This will create a new cycle of length $m+n$ with the same incidence set. If $m$ and $n$ are both even, then $m+n$ is even and the final sign changes due to the loss of an even cycle. If $m$ is even and $n$ is odd or vice versa, then $m+n$ is odd and the final sign changes due to the loss of an even cycle. If $m$ and $n$ are both odd then $m+n$ is even and the final sign changes due to the gain of an even cycle. \qed 
\end{proof}

Using what we now know about non-edge monic tail-equivalence classes we can restate Theorem \ref{t:Main}.

\begin{theorem}[Total-minor Polynomial v2]
Let $G$ be an oriented hypergraph with Laplacian matrix $\mathbf{L}_{G}$. Then, 
\begin{align*}
    \chi(\mathbf{L}_{G},\mathbf{X}) = \dsum\limits_{ [\mathbf{u},\mathbf{w}]}\left( \dsum\limits_{\substack{c \in \mathfrak{\hat{M}}_{\neq 0}(G^0;\mathbf{u},\mathbf{w}) \\ }} \csgn(c) \right) \dprod\limits_{i}x_{u_{i},w_{i}}.
\end{align*}
\label{t:Main2}
\end{theorem}
\begin{proof}

Let $\mathcal{N}$ be the set of non-edge-monic tail equivalence (head equivalence) classes and let $\mathcal{M}$ be the set of edge-monic tail equivalence (head equivalence) classes.

We have
\begin{align*}
    \mathfrak{C} = \mathfrak{N}(G) \cup \mathfrak{M}(G) = \bigcup_{N \in \mathcal{N}} N \cup \bigcup_{M \in \mathcal{M}} M
\end{align*}

Thus, the inner sum can be rewritten as
\begin{align*}
\dsum\limits_{\substack{c \in \mathfrak{\hat{C}}_{\neq 0}(G^{0};\mathbf{u},\mathbf{w}) \\ }} (-1)^{ec(\check{c})+nc(c)+bs(c)} &= \dsum\limits_{\substack{c \in \mathfrak{\hat{N}}_{\neq 0}(G^{0};\mathbf{u},\mathbf{w}) \\ }} (-1)^{ec(\check{c})+nc(c)+bs(c)} \\
&+ \dsum\limits_{\substack{c \in \mathfrak{\hat{M}}_{\neq 0}(G^{0};\mathbf{u},\mathbf{w}) \\ }} (-1)^{ec(\check{c})+nc(c)+bs(c)}
\end{align*}

Rewriting as sums over tail-equivalence classes, 
\begin{align*}
= \dsum\limits_{\substack{N \in \mathcal{\hat{N}}_{\neq 0}(G^{0};\mathbf{u},\mathbf{w}) \\ }} \dsum\limits_{c \in N} (-1)^{ec(\check{c})+nc(c)+bs(c)} + \dsum\limits_{\substack{M \in \mathcal{\hat{M}}_{\neq 0}(G^{0};\mathbf{u},\mathbf{w}) \\ }} \dsum\limits_{c \in M} (-1)^{ec(\check{c})+nc(c)+bs(c)}.
\end{align*}

By Theorem \ref{t:contcancel} we have
\begin{align*}
&= \dsum\limits_{\substack{N \in \mathcal{\hat{N}}_{\neq 0}(G^{0};\mathbf{u},\mathbf{w}) \\ }}  0 +  \dsum\limits_{\substack{M \in \mathcal{\hat{M}}_{\neq 0}(G^{0};\mathbf{u},\mathbf{w}) \\ }} \dsum\limits_{c \in M} (-1)^{ec(\check{c})+nc(c)+bs(c)} \\
&= \dsum\limits_{\substack{c \in \mathfrak{\hat{M}}_{\neq 0}(G^{0};\mathbf{u},\mathbf{w}) \\ }} (-1)^{ec(\check{c})+nc(c)+bs(c)}, 
\end{align*}  
which gives the desired result. \qed
\end{proof}

\subsection{Main Theorem Examples}

Below is the total-minor polynomial for $G_2$ from Figure \ref{fig:OHs} and the three bolded terms are calculated using their contributor maps.
\begin{align*}
\det(\mathbf{X - L}_{G_{2}}) &= -1 + 2x_{11} - 3x_{12} - 2x_{13} \mathbf{ - 3x_{21}} + 6x_{22} + 4x_{23} - 2x_{31} + 4x_{32} + 3x_{33} \\
& \mathbf{- 3x_{11}x_{22}} + 3x_{12}x_{21} \mathbf{- 2x_{11}x_{23}} + 2x_{13}x_{21} - 2x_{11}x_{32} + 2x_{12}x_{31} - 2x_{11}x_{33} \\
&+ 2x_{13}x_{31} + x_{12}x_{33} - x_{13}x_{32} + x_{21}x_{33} - x_{23}x_{31} - 2x_{22}x_{33} + 2x_{23}x_{32} \\
&+ x_{11}x_{22}x_{33} - x_{11}x_{23}x_{32} - x_{12}x_{21}x_{33} + x_{12}x_{23}x_{31} + x_{13}x_{21}x_{32} - x_{13}x_{22}x_{31}.
\end{align*}

First, consider all contributors in ${\mathfrak{\hat{C}}}(G_2;(v_1,v_2),(v_1,v_2))$. That is, contributors where $v_1 \mapsto v_1$ and $v_2 \mapsto v_2$ exists, but these backsteps have been removed in order to choose the $x_{1,1}$ and $x_{2,2}$ monomials. The remaining reduced contributors these are shown in Figure \ref{fig:TMP1}. The $\csgn$ of each is $-1$ as they each contain exactly one backstep, no negative components, and their original contributor was from an identity, so they contained no even cycles. 

\begin{figure}[H]
    \centering
    \includegraphics[scale=1]{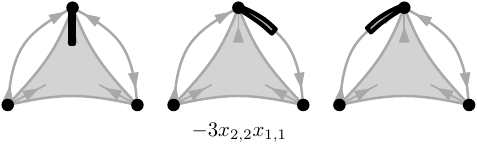}
    \caption{Example 1}
    \label{fig:TMP1}
\end{figure}

Next, consider all contributors in ${\mathfrak{\hat{C}}}(G_2;(v_1,v_2),(v_1,v_3))$. That is, contributors where $v_1 \mapsto v_1$ and $v_2 \mapsto v_3$ exists. Here we lose the  $(v_1,v_1)$-backstep and the $(v_2,v_3)$-adjacency as the monomials $x_{1,1}$ and $x_{2,3}$ are chosen instead. The remaining reduced contributors these are shown in Figure \ref{fig:TMP2}. The $\csgn$ of each is $-1$ as they contain zero backsteps, no negative components and their original contributor contained a single $2$-cycle (regardless of how it was reduced). 

\begin{figure}[H]
    \centering
    \includegraphics[scale=1]{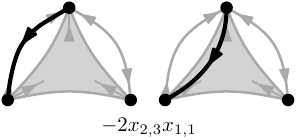}
    \caption{Example 2}
    \label{fig:TMP2}
\end{figure}

Finally, consider all contributors in ${\mathfrak{\hat{C}}}(G_2;(v_2),(v_1))$. That is, contributors where $v_2 \mapsto v_1$ and is removed in order to create the $x_{2,1}$ monomial. The remaining reduced contributors these are shown in Figure \ref{fig:TMP3}. The $\csgn$ of the top three are $-1$, the $\csgn$ of the first two on the bottom are also $-1$, while the $\csgn$ of the last two on the bottom are $+1$. Thus producing a coefficient value of $-5 + 2 = -3$.

\begin{figure}[H]
    \centering
    \includegraphics[scale=1]{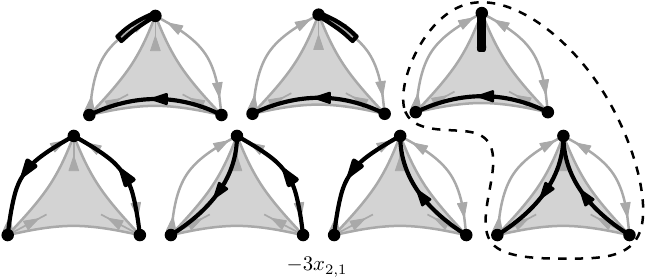}
    \caption{Example 3}
    \label{fig:TMP3}
\end{figure}

However, the rightmost two contributors in the dashed circle belong to the same non-edge-monic tail-equivalence class, and they must cancel by Theorem \ref{t:Main2}. The removal of these two contributors would produce ${\mathfrak{\hat{M}}}(G_2;(v_2),(v_1))$, the optimal collection of contributors to determine the coefficient as $-4 + 1 = -3$.

\section{Duality and classical results}

\subsection{Duality of contributors}

We now ask when is the incidence-dual of a contributor again a contributor? The \emph{contributor-dual} of a contributor $c$ in $G$ is the contributor $c^*$ in $G^*$ obtained by the mappings of $\overrightarrow{P}_{1}$ into $G^*$ that preserves (and reverses) the cyclic incidence sequences of each local component. A few clarifying observations: (1) Dualizing both the $\overrightarrow{P}_{1}$'s and $G$ produces the same set of contributors with edge-to-edge steps and $G^*$; (2) tail-equivalence classes become head-equivalence classes in $G^*$, reversing the incidence sequences allow for streamlining interpretation through tail-equivalence; (3) not every contributor has a dual. 

Contributor duality is best understood as shifting the $V \rightarrow V$ contributors to $E \rightarrow E$ by initial map for tail-equivalence over a single incidence (half of an adjacency), then dualizing so the edges become vertices. This shift is seen below in bold:
\begin{align*}
    \mathbf{V \rightarrow E \rightarrow V} \rightarrow E \\
    V \rightarrow \mathbf{E \rightarrow V \rightarrow E}.
\end{align*}
Observe that the initial $V \rightarrow E$ tail maps are shifted to the $E \rightarrow V$ head maps. Moreover, these are extended to contributors via another $V \rightarrow E$ map. A contributor dual requires both maps to produce permutation analogs that are also incidence preserving, this is shown in Figure \ref{fig:contdual}.

\begin{figure}[H]
    \centering
    \includegraphics[scale=1]{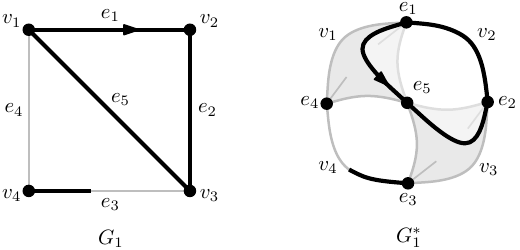}
    \caption{A contributor in $G_1$ and its dual in $G^*_1$. The arrow indicates the direction of travel of the respective path maps.}
    \label{fig:contdual}
\end{figure}

Since any single missing adjacency in a reduced contributor $c$ would become a missing co-adjacency in $c^*$, and no longer meet the definition of ``contributor,'' to ensure duality we must reduce every adjacency in a circle or leave the circle alone, given us the following Lemma.

\begin{lemma}\label{l:InCircle}
If a reduced contributor $c$ in $G$ has a contributor-dual, then every adjacency of $c$ is in a circle.
\end{lemma}

Edge-monicness plays an important role in contributor duality as well.

\begin{lemma}\label{l:CDualsAreMonic}
    A reduced contributor $c$ in $G$ that has a contributor-dual must be edge-monic.
\end{lemma} 
\begin{proof}
    If not, then a repeated edge would produce two tails at the same vertex in $c^*$, preventing an associated permutation in $G^{*}$. \qed
\end{proof}

Let $\mathfrak{\hat{M}}_G^{\circ}$ be the set of edge-monic reduced contributors of $G$ where every adjacency is in a circle --- that is, no component of a contributor is a path unless it is a single vertex. 

\begin{lemma}\label{l:DiagHasDual}
    Every contributor of $\mathfrak{\hat{M}}_G^{\circ}$ has a contributor dual.
\end{lemma}
\begin{proof}    
    Let $c$ be an edge-monic reduced contributor of $G$ with every adjacency in a circle. Thus, every component of $c$ is a circle or backstep; hence, self-dual. Thus, there is a contributor dual $c^{*}$ in $G^{*}$. \qed
\end{proof}

We now demonstrate that contributor duals preserve contributor sign.

\begin{lemma}
Given an oriented hypergraph $G$, its dual $G^{*}$, and reduced contributors $c$ and $c^{*}$ respectively, we have $\csgn(c)=\csgn(c^{*})$.
\label{l:samesign}
\end{lemma}
\begin{proof}
By definition there is a bijection between the components of $c$ and $c^{*}$ where the incidence sequence for each component-pairing forms the same circle. Backsteps dualize to the same incidence sequence, thus, having no affect on $bs(c)$. For the circles left intact in $\check{c}$ and $\check{c}^{*}$ each pair of dual circles must have the same length. Thus, having no effect on $ec(\check{c})$. Moreover, any circle in both $\check{c}$ and $\check{c}^{*}$ have the same incidence sequence, so they have the same circle sign, having no effect on $nc(c)$. \qed
\end{proof}

Leaving circles untouched (with the exception of loops) immediately reduces the duality problem to only backsteps and loops, which occur along the main diagonal, giving the traditional characteristic polynomial. The elements of $\mathfrak{\hat{M}}_G^{\circ}$ where only backsteps and loops are reduced to produce monomials are referred to as \emph{diagonally reduced} and is denoted $\mathfrak{\hat{D}}_G$.

As we have seen in Theorem \ref{t:Main2}, sets of reduced contributors provide a way to calculate the coefficients of each monomial. In order to track each monomial through duality let the \emph{monomial of a $[\mathbf{u},\mathbf{w}]$-reduced contributor $c$} be
\begin{align*}
    m_c = \csgn(c)\dprod\limits_{i}x_{u_{i},w_{i}}.
\end{align*}

In Figure \ref{fig:duals1} we see each of the $8$ non-cancellative $[v_2,v_2]$-reduced contributors from Figure \ref{fig:v2v2l} that produce spanning trees along with their contributor-duals. Each contributor also depicts their associated monomial via the assumed missing backsteps for diagonally reduced contributors.

\begin{figure}[H]
    \centering    \includegraphics{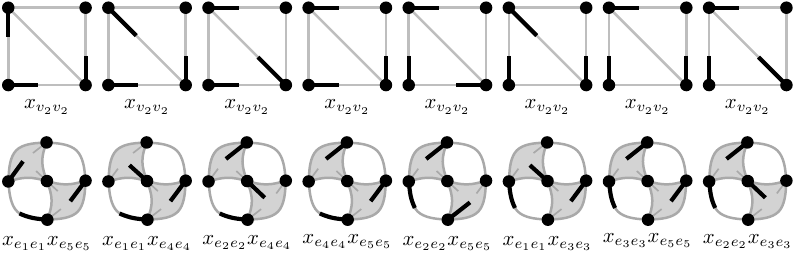}
    \caption{Contributor duals and their associated monomials.}
    \label{fig:duals1}
\end{figure} 

Observe that in Figure \ref{fig:duals1} the unused edges in the contributors of $G$ on the top row become the monomials in the dual. The same applies when moving from $G^*$ to $G$.

\subsection{A Classical Result}

We now focus on the non-$0$-isospectrality of the Laplacian and its dual (see Equation \ref{eq:1}) and providing further combinatorial insight as to why this simple relationship holds. Since the focus is on the main diagonal relationship we will use $\mathbf{x}$ to indicate a vector of monomials corresponding to the main diagonal of $\mathbf{X}$, subscripts of $V$ or $E$ will indicate which set is regarded as vertices so we can easily track the two characteristic polynomials.

We apply the main theorem, Theorem \ref{t:Main2}, over only diagonally reduced contributors to provide an expanded version of the traditional characteristic polynomials from Equation \ref{eq:1}. Here, the location of each $x$ is tracked across each monomial. Thus, for $G_1$ we have 

\begin{equation}
\begin{aligned} 
\chi (\mathbf{L}_{G_{1}},\mathbf{x}_{V}) &= \det(I\mathbf{x}_{V}-\mathbf{L}_{G_{1}})= \mathbf{- 8x_{v_{2}v_{2}}} - 8x_{v_{3}v_{3}} - 8x_{v_{4}v_{4}} - 8x_{v_{1}v_{1}} \\ &+ 5x_{v_{1}v_{1}}x_{v_{2}v_{2}} + 4x_{v_{1}v_{1}}x_{v_{3}v_{3}} + 5x_{v_{1}v_{1}}x_{v_{4}v_{4}} + 5x_{v_{2}v_{2}}x_{v_{3}v_{3}} + 8x_{v_{2}v_{2}}x_{v_{4}v_{4}} \\ &+ 5x_{v_{3}v_{3}}x_{v_{4}v_{4}} - 2x_{v_{1}v_{1}}x_{v_{2}v_{2}}x_{v_{3}v_{3}} - 3x_{v_{1}v_{1}}x_{v_{2}v_{2}}x_{v_{4}v_{4}} - 2x_{v_{1}v_{1}}x_{v_{3}v_{3}}x_{v_{4}v_{4}} \\& - 3x_{v_{2}v_{2}}x_{v_{3}v_{3}}x_{v_{4}v_{4}} + x_{v_{1}v_{1}}x_{v_{2}v_{2}}x_{v_{3}v_{3}}x_{v_{4}v_{4}}. \label{eq:3}
\end{aligned}
\end{equation}
Observe that the bolded coefficient of $- 8x_{v_{2}v_{2}}$ is determined from the $\csgn$ of the contributors from Figure \ref{fig:v2v2l}. We will prove that these $8$ contributors scatter throughout the dual according to their dual as shown in Figure \ref{fig:duals1}.

Since every diagonally reduced contributor has a contributor dual by Lemma \ref{l:DiagHasDual}, define the \emph{extended monomial} of a diagonally reduced contributor $c$ as the product of the monomials $m_c m_{c^*}$. Clearly, $c$ and $c^*$ have the same extended monomial. The extended monomial of the left-most contributor dual pair in Figure \ref{fig:duals1} is $x_{v_2 v_2}x_{e_1 e_1}x_{e_5 e_5}$.

\begin{lemma}\label{l:samemonomial}
    If $c$ and $c^{*}$ are a dual pair of diagonally reduced contributors, then they have the same extended monomial.
\end{lemma}

The \emph{extended characteristic polynomial} of $G$ is the sum of the extended monomials. Let $\chi^{e} (\mathbf{L}_{G},\mathbf{x}_{V \cup E})$ be the extended polynomial of the Laplacian and $\chi^{e} (\mathbf{L}_{G}^{*},\mathbf{x}_{E \cup V})$ be the extended polynomial of the dual Laplacian. From Lemma \ref{l:samemonomial} we immediately have the following.

\begin{proposition}
\label{p:equalextpoly}
For any oriented hypergraph $G$, $\chi^{e} (\mathbf{L}_{G},\mathbf{x}_{V \cup E}) = \chi^{e} (\mathbf{L}_{G}^{*},\mathbf{x}_{E \cup V})$.
\end{proposition}

However, we need to demonstrate that the extended characteristic polynomial is obtainable using a single set of diagonally reduced contributors either from $G$ or $G^*$, but not both. Define the \emph{vertex complement of contributor $c$} as the set of vertices with no head or tail present in $c$, i.e. those which were reduced, denoted by $\overline{V}(c)$. Similarly, $\overline{E}(c)$ denotes the unused edges. 

\begin{theorem}
\label{t:extpolysum}
    Let $\mathfrak{\hat{D}}_G$ be the set of diagonally reduced contributors of $G$. Then the extended characteristic polynomial is given by
    \begin{align*}
        \chi^{e} (\mathbf{L}_{G},\mathbf{x}_{V \cup E}) &= \dsum\limits_{c \in \mathfrak{\hat{D}}_G} sgn(c) \dprod\limits_{v \in \overline{V}(c)} x_{vv}\dprod\limits_{e \in \overline{E}(c)} x_{ee} \\
        &= \dsum\limits_{c^* \in \mathfrak{\hat{D}}_{G^{*}}} sgn(c^*) \dprod\limits_{e \in \overline{E}(c)} x_{ee} \dprod\limits_{v \in \overline{V}(c)} x_{vv} = \chi^{e} (\mathbf{L}^{*}_{G},\mathbf{x}_{E \cup V}).
    \end{align*}
\end{theorem}
\begin{proof}
By Lemmas \ref{l:samesign} and \ref{l:samemonomial} we know the middle sums are equal. By Theorem \ref{t:Main2} and Lemma \ref{l:DiagHasDual} we have the first and last equality. \qed
\end{proof}

We now characterize the deviation between monomial degrees of dual contributors as well as their deviation from their maximum degree.

\begin{lemma}
\label{l:diffisconstant}
Let $c$ and $c^*$ be dual contributor pairs from oriented hypergraph $G$ with monomials $m_c$ and $m_{c^*}$, then $\left| deg(m_c) - deg(m_{c^*}) \right| = \left| \left|E\right|-\left|V\right| \right|$.
\end{lemma}
\begin{proof}
Let $G$ have $|V|$ many vertices and $|E|$ many edges, and let $c$ and $c^{*}$ be a pair of dual contributors. Observe that $c$ and $c^{*}$ contain the same number of $\overrightarrow{P}_{1}$ maps, call it $n$. Thus, the monomial-degree of $c$ in $\chi (\mathbf{L}_{G},\mathbf{x}_{V})$ is $|V|-n$ and the monomial-degree of $c^{*}$ in $\chi (\mathbf{L}_{G}^{*},\mathbf{x}_{E})$ is $|E|-n$. Taking their difference yields the result. \qed
\end{proof}

\begin{lemma}
\label{l:samenumofxs}
Let $m_c$ and $m_{c^*}$ be monomials in $\chi (\mathbf{L}_{G},\mathbf{x}_{V})$ and $\chi (\mathbf{L}_{G}^{*},\mathbf{x}_{E})$ respectively. Then $|V|-deg(m_c)=|E|-deg(m_{c^*})$. 
\end{lemma}
\begin{proof}
Observe that $|V|-deg(m)$ is the number of $\overrightarrow{P}_{1}$ maps into $G$ creating $c$ and $|E|-deg(m^{*})$ is the number of $\overrightarrow{P}_{1}$ maps into $G^{*}$ and since $c$ and $c^{*}$ are a contributor dual pair, the number of $\overrightarrow{P}_{1}$ maps into $G$ and $G^{*}$ are the same. \qed
\end{proof}

Observe that Lemma \ref{l:samenumofxs} implies that given monomials $m_c$ and $m_{c^*}$ in $\chi (\mathbf{L}_{G},\mathbf{x}_{V})$ and $\chi (\mathbf{L}_{G}^{*},\mathbf{x}_{E})$ respectively, that the number of missing $x_{v_{i}v_{i}}$ to obtain the maximum degree of 
$|V|$ in $m_c$ is equal to the number of missing $x_{e_{i}e_{i}}$ to obtain the maximum degree of 
$|E|$ in $m_{c^*}$.

\begin{theorem}
Given an oriented hypergraph $G$ and its dual, $G^{*}$, $x^{|E|}\chi (\mathbf{L}_{G},x) =x^{|V|}\chi (\mathbf{L}_{G}^{*},x)$.
\label{t:equal}
\end{theorem}
\begin{proof}
Using Theorem \ref{t:extpolysum} we have that $\chi^{e} (\mathbf{L}_{G},\mathbf{x}_{V \cup E})=\chi^{e} (\mathbf{L}_{G}^{*},\mathbf{x}_{E \cup V})$ and they can be determined by either set of diagonally reduced contributors. By Lemma \ref{l:samenumofxs} we also have $|V|-deg(m_c)=|E|-deg(m_{c^*})$ for each dual monomial pair. Thus, multiplying each monomial pair by an un-indexed indeterminant $x^{|V|-deg(m_c)} = x^{|E|-deg(m_{c^*})}$ preserves equality and introduces all missing $x$ variables to form a temporary extended-extended characteristic polynomial.

Now consider the extended-extended version of $\chi^{e} (\mathbf{L}_{G},\mathbf{x}_{V \cup E})$ and drop the subscripts on every $x_{e_{i}e_{i}}$ to only be $x$ giving $x^{|E|}\chi (\mathbf{L}_{G},\mathbf{x}_{V})$. Similarly, we transform the $\chi^{e} (\mathbf{L}_{G}^{*},\mathbf{x}_{E \cup V})$ side by dropping the subscript on every $x_{v_{i}v_{i}}$ to give $x$ producing $x^{|V|}\chi (\mathbf{L}_{G}^{*},\mathbf{x}_{E})$. 

Finally, drop all the remaining indices to produce the traditional characteristic polynomial relationship. \qed
\end{proof}

Figure \ref{fig:flowofcont} illustrates how the first dual contributor pair from Figure \ref{fig:duals1} is parsed according to Theorem \ref{t:equal}.

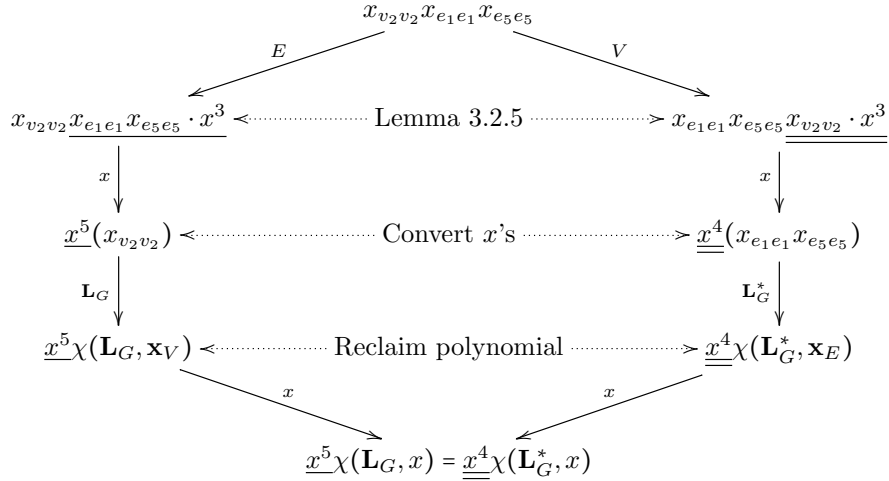
\begin{figure}[H]
    \centering
\begin{align*}
\xymatrix{
   &  {x_{v_{2}v_{2}}}{x_{e_{1}e_{1}}x_{e_{5}e_{5}}} \ar[dl]_{{E}}\ar[dr]^{{V}}  &  \\
    {x_{v_{2}v_{2}}}\underline{{x_{e_{1}e_{1}}x_{e_{5}e_{5}} \cdot x^3}} \ar[d]_{{x}} & \ar@{..>}[l] \text{Lemma \ref{l:samenumofxs}} \ar@{..>}[r] & {x_{e_{1}e_{1}}x_{e_{5}e_{5}}}\underline{\underline{{x_{v_{2}v_{2}}} \cdot x^3 }} \ar[d]_{{x}}  \\
 \underline{{x^{5}}}{(x_{v_{2}v_{2}})} \ar[d]_{{\mathbf{L}_{G}}}& \ar@{..>}[l] \text{Convert $x$'s} \ar@{..>}[r] & \underline{\underline{{x^4}}}{{(x_{e_{1}e_{1}}x_{e_{5}e_{5}})}} \ar[d]_{{\mathbf{L}_{G}^{*}}}\\
  \underline{{x^{5}}}{\chi (\mathbf{L}_{G},\mathbf{x}_{V})} \ar[dr]^{{x}}  & \ar@{..>}[l] \text{Reclaim polynomial} \ar@{..>}[r]  &  \underline{\underline{{x^{4}}}}{\chi (\mathbf{L}_{G}^{*},\mathbf{x}_{E})} \ar[dl]_{{x}} \\
    & \underline{x^{5}}\chi (\mathbf{L}_{G},x) = \underline{\underline{x^{4}}}\chi (\mathbf{L}_{G}^{*},x)  &  \\
}
\end{align*}
    \caption{Each dual contributor pair allows for the inclusion of the same number of missing indeterminants that that provide a constant factor our of each term.}
    \label{fig:flowofcont}
\end{figure}

\subsection{Putting it all together}

Recall that the characteristic polynomial in Equation \ref{eq:3} is a location-distinguished version of the traditional characteristic polynomial from Equation \ref{eq:1} for $G_1$. The traditional characteristic polynomial of $G_1^*$ appears in Equation \ref{eq:2}, while the location-distinguished version appears in Equation \ref{eq:4}. The coefficients of the $x_{e_{i}e_{i}}x_{e_{j}e_{j}}$ terms are all $-4$, a single but written as $-(3+\mathbf{{1}})$ to highlight where each dual contributors from Figure \ref{fig:duals1} occurs.

\begin{equation}
\begin{aligned} \label{eq:4}
\chi (\mathbf{L}_{G_{1}}^{*},\mathbf{x}_{E})&=\det(I\mathbf{x}_{E}-\mathbf{L}_{G_{1}}^{*})=- (3+\mathbf{{1}})x_{e_{1}e_{1}}x_{e_{4}e_{4}} - (3+\mathbf{{1}})x_{e_{2}e_{2}}x_{e_{3}e_{3}} \\ &- (3+\mathbf{{1}})x_{e_{1}e_{1}}x_{e_{5}e_{5}} - (3+\mathbf{{1}})x_{e_{2}e_{2}}x_{e_{4}e_{4}} - (3+\mathbf{{1}})x_{e_{2}e_{2}}x_{e_{5}e_{5}} \\ & - (3+\mathbf{{1}})x_{e_{3}e_{3}}x_{e_{5}e_{5}} - (3+\mathbf{{1}})x_{e_{4}e_{4}}x_{e_{5}e_{5}} - (3+\mathbf{{1}})x_{e_{1}e_{1}}x_{e_{3}e_{3}} \\ & + 3x_{e_{1}e_{1}}x_{e_{2}e_{2}}x_{e_{3}e_{3}} + 3x_{e_{1}e_{1}}x_{e_{2}e_{2}}x_{e_{4}e_{4}} + 3x_{e_{1}e_{1}}x_{e_{2}e_{2}}x_{e_{5}e_{5}} \\ &+ 3x_{e_{1}e_{1}}x_{e_{3}e_{3}}x_{e_{4}e_{4}} + 4x_{e_{1}e_{1}}x_{e_{3}e_{3}}x_{e_{5}e_{5}} + 3x_{e_{2}e_{2}}x_{e_{3}e_{3}}x_{e_{4}e_{4}} \\ &+ 3x_{e_{1}e_{1}}x_{e_{4}e_{4}}x_{e_{5}e_{5}} + 3x_{e_{2}e_{2}}x_{e_{3}e_{3}}x_{e_{5}e_{5}} + 4x_{e_{2}e_{2}}x_{e_{4}e_{4}}x_{e_{5}e_{5}} \\ &+ 3x_{e_{3}e_{3}}x_{e_{4}e_{4}}x_{e_{5}e_{5}} - 2x_{e_{1}e_{1}}x_{e_{2}e_{2}}x_{e_{3}e_{3}}x_{e_{4}e_{4}} - 2x_{e_{1}e_{1}}x_{e_{2}e_{2}}x_{e_{3}e_{3}}x_{e_{5}e_{5}} \\ &- 2x_{e_{1}e_{1}}x_{e_{2}e_{2}}x_{e_{4}e_{4}}x_{e_{5}e_{5}} - 2x_{e_{1}e_{1}}x_{e_{3}e_{3}}x_{e_{4}e_{4}}x_{e_{5}e_{5}} \\ &- 2x_{e_{2}e_{2}}x_{e_{3}e_{3}}x_{e_{4}e_{4}}x_{e_{5}e_{5}} + x_{e_{1}e_{1}}x_{e_{2}e_{2}}x_{e_{3}e_{3}}x_{e_{4}e_{4}}x_{e_{5}e_{5}}. 
\end{aligned} 
\end{equation}

The coefficient of the bold $\mathbf{-8 x_{v_2 v_2}}$ from $G_1$ in Equation \ref{eq:3} corresponds to the top contributors from Figure \ref{fig:duals1}. The corresponding dual contributors on the bottom of Figure \ref{fig:duals1} produce the bolded $\mathbf{1}$'s in Equation \ref{eq:4}. We see that the $-8$ coefficient is uniformly scattered across minors of the dual. Interestingly, this implies that the rooted first minors of a graph that produce the matrix-tree theorem (via extending the contributors) do not have spanning tree analogs in the dual. In fact, it is the uniqueness of the other end of a $2$-edge that enforces this. Thus, while we are able to narrow in on the optimal set of contributors to determine any integer minor, we cannot expect matrix-tree-type theorems and will need a uniquely hypergraphic interpretation to move forward.


\begin{thebibliography}{10}

\bibitem{Sim1}
Francesco Belardo and Slobodan Simi{\'c}.
\newblock On the laplacian coefficients of signed graphs.
\newblock {\em Linear Algebra and its Applications}, 475:94--113, 2015.

\bibitem{Seth1}
Seth Chaiken.
\newblock A combinatorial proof of the all minors matrix tree theorem.
\newblock {\em SIAM Journal on Algebraic Discrete Methods 3}, 3:319--329, 1982.

\bibitem{OHSachs}
G.~Chen, V.~Liu, E.~Robinson, L.~J. Rusnak, and K.~Wang.
\newblock A characterization of oriented hypergraphic laplacian and adjacency matrix coefficients.
\newblock {\em Linear Algebra and its Applications}, 556:323 -- 341, 2018.

\bibitem{CLARK2022354}
Gregory~J. Clark and Joshua~N. Cooper.
\newblock Applications of the harary-sachs theorem for hypergraphs.
\newblock {\em Linear Algebra and its Applications}, 649:354--374, 2022.

\bibitem{BM2}
Michele Conforti and G{\'e}rard Cornu{\'e}jols.
\newblock Balanced 0, $\pm{1}$ matrices, bicoloring and total dual integrality.
\newblock {\em Math. Programming}, 71:249--258, 1995.

\bibitem{DBM}
Michele Conforti, G{\'e}rard Cornu{\'e}jols, and M.~R. Rao.
\newblock Decomposition of balanced matrices.
\newblock {\em J. Combin. Theory Ser. B}, 77(2):292–406, 1999.

\bibitem{BM}
Michele Conforti, G{\'e}rard Cornu{\'e}jols, and Kristina Vu{\v{s}}kovi{\'c}.
\newblock Balanced matrices.
\newblock {\em Discrete Math.}, 306(19-20):2411--2437, 2006.

\bibitem{SGBook}
D.~Cvetkovic, M.~Doob, and H.~Sachs.
\newblock {\em Spectra of Graphs: Theory and Applications, 3rd revised and enlarged edition}.
\newblock Oxford Science Publications. Wiley-VCH, 1998.

\bibitem{Reff6}
Luke Duttweiler and Nathan Reff.
\newblock Spectra of cycle and path families of oriented hypergraphs.
\newblock {\em Linear Algebra Appl.}, 578:251--271, 2019.

\bibitem{MR0267898}
Jack Edmonds and Ellis~L. Johnson.
\newblock Matching: {A} well-solved class of integer linear programs.
\newblock In {\em Combinatorial {S}tructures and their {A}pplications ({P}roc. {C}algary {I}nternat., {C}algary, {A}lta., 1969)}, pages 89--92. Gordon and Breach, New York, 1970.

\bibitem{IH1}
W.~Grilliette and L.~J. Rusnak.
\newblock Incidence hypergraphs: The categorical inconsistency of set-systems and a characterization of quiver exponentials.
\newblock {\em ArXiv:1805.07670 [math.CO]}, 2018.

\bibitem{IH2}
Will Grilliette, Josephine Reynes, and Lucas~J. Rusnak.
\newblock Incidence hypergraphs: Injectivity, uniformity, and matrix-tree theorems.
\newblock {\em Linear Algebra and its Applications}, 634:77--105, 2022.

\bibitem{harary1962determinant}
Frank Harary.
\newblock The determinant of the adjacency matrix of a graph.
\newblock {\em Siam Review}, 4(3):202--210, 1962.

\bibitem{Reff7}
Ouail Kitouni and Nathan.
\newblock Lower bounds for the laplacian spectral radius of an oriented hypergraph.
\newblock {\em Australasian Journal of Combinatorics}, 74(3):408--422, 2019.

\bibitem{Mulas3}
R.~Mulas.
\newblock Sharp bounds for the largest eigenvalue of the normalized hypergraph laplace operator.
\newblock {\em Mathematica Notes, to appear}, 2020.

\bibitem{Mulas1}
R.~Mulas.
\newblock Spectral classes of hypergraphs.
\newblock {\em ArXiv:2007.04273 [math.CO]}, 2020.

\bibitem{AH1}
N.~Reff and L.J. Rusnak.
\newblock An oriented hypergraphic approach to algebraic graph theory.
\newblock {\em Linear Algebra and its Applications}, 437(9):2262--2270, 2012.

\bibitem{Reff2}
Nathan Reff.
\newblock Spectral properties of oriented hypergraphs.
\newblock {\em Electron. J. Linear Algebra}, 27:373--391, 2014.

\bibitem{Reff3}
Nathan Reff.
\newblock Intersection graphs of oriented hypergraphs and their matrices.
\newblock {\em Australas. J. Combin.}, 65:108--123, 2016.

\bibitem{Reff5}
Nathan Reff and Howard Skogman.
\newblock A connection between {H}adamard matrices, oriented hypergraphs and signed graphs.
\newblock {\em Linear Algebra Appl.}, 529:115--125, 2017.

\bibitem{OHMTT}
L.~J. Rusnak, E.~Robinson, M.~Schmidt, and P.~Shroff.
\newblock Oriented hypergraphic matrix-tree type theorems and bidirected minors via boolean ideals.
\newblock {\em J Algebr Comb}, 49(4):461–--473, 2019.

\bibitem{OH1}
L.J. Rusnak.
\newblock Oriented hypergraphs: Introduction and balance.
\newblock {\em Electronic J. Combinatorics}, 20(3)(\#P48), 2013.

\bibitem{SGKirch}
Lucas~J Rusnak, Josephine Reynes, Skyler~J Johnson, and Peter Ye.
\newblock Generalizing kirchhoff laws for signed graphs.
\newblock {\em AUSTRALASIAN JOURNAL OF COMBINATORICS}, 81(3):388--411, 2021.

\bibitem{OHHada}
Lucas~J. Rusnak, Josephine Reynes, Russell Li, Eric Yan, and Justin Yu.
\newblock The determinant of {±1}-matrices and oriented hypergraphs.
\newblock {\em Linear Algebra and its Applications}, 702:161--178, 2024.

\bibitem{sachs1966teiler}
Horst Sachs.
\newblock {\"U}ber teiler, faktoren und charakteristische polynome von graphen.
\newblock {\em Teil I. Wiss. Z. TH Ilmenau}, 12:7--12, 1966.

\bibitem{SG}
Thomas Zaslavsky.
\newblock Signed graphs.
\newblock {\em Discrete Appl. Math.}, 4(1):47--74, 1982.
\newblock MR 84e:05095a. Erratum, ibid., 5 (1983), 248. MR 84e:05095b.

\bibitem{OSG}
Thomas Zaslavsky.
\newblock Orientation of signed graphs.
\newblock {\em European J. Combin.}, 12(4):361--375, 1991.

\end{thebibliography}
\end{document}